\documentclass[10pt,leqno]{amsart}
\usepackage{graphicx}
\usepackage{indentfirst,csquotes}

\usepackage{amssymb,amsthm,amsmath,mathtools}
\usepackage{xcolor,paralist,hyperref,fancyhdr,etoolbox}
\usepackage[english]{babel}
\newtheorem{theorem}{Theorem}[]

\newtheorem{lemma}[theorem]{Lemma}

\theoremstyle{remark}
\newtheorem*{remark}{Remark}

\hypersetup{ colorlinks=true, linkcolor=black, filecolor=black, urlcolor=black }

\usepackage[hyperref,doi,citestyle=ieee,backend=bibtex,sorting=nyt]{biblatex}
\bibliography{bibliography}
\AtEveryBibitem{
 \clearfield{date}
 \clearfield{eprint}
 \clearfield{isbn}
 \clearfield{issn}
 \clearlist{location}
 \clearfield{month}
 \clearfield{doi}

 \ifentrytype{book}{}{
  \clearlist{publisher}
  \clearname{editor}
  \clearfield{booktitle}
 }
}
\usepackage{comment}
\begin{document}
\title{On Hollenbeck-Verbitsky conjecture for  \(4/3 <  p <  2\) and reverse Riesz-type inequalities for \(2<p<4\)} 
\author{Vladan Jaguzović}
\address{ Faculty of Natural Sciences and Mathematics \newline\indent
University of Banja Luka\newline\indent
Mladena Stojanovića 2\newline\indent
78000 Banja Luka \newline\indent
Republic of Srpska\newline\indent
Bosnia and Herzegovina
}
\email{vladan.jaguzovic@pmf.unibl.org}

\begin{NoHyper} 
  \let\thefootnote\relax
\end{NoHyper}
\begin{abstract}
Let \(P_+\) be the Riesz's projection operator and let \(P_-=I-P_+.\) We find the best upper estimates of the  expression \(\left\lVert \left( \left\lvert P_+f \right\rvert ^s +  \left\lvert P_-f \right\rvert ^s \right) ^{1/s}    \right\rVert _p  \) in terms of Lebesgue p-norm of the function \(f \in L^p(\mathbf{T})\) for \(p \in (4/3,2)\) and \(0 < s \leq  \frac{p}{p-1},\) thus extending results from  \cite{Melentijevic_2022} and \cite{Melentijevic_2023}, where the mentioned range is not considered. Also, we find the best lower estimates of the same quantities for \(p \in (2,4)\) and \(s \geq \frac{p}{p-1},\) thus extending results from \cite{melentijevic-reverse-2025}.
\end{abstract} 
\maketitle \pagestyle{myheadings} \markboth{VLADAN JAGUZOVIĆ}{RIESZ-TYPE INEQUALITIES}
\section{Introduction and main result}
For \(1 \leq  p \leq  \infty\) we denote by \(L^p(\mathbb{T})\) the Lebesgue space on the unit circle \(\mathbb{T}\) in the complex plane \(\mathbb{C}.\)  By \(H^p(\mathbb{T})\) we denote a subspace of all functions for which negative Fourier coefficients are equal zero, i.e., 
 \[
   f \in H^p(\mathbb{T}) \Leftrightarrow f \in L^p(\mathbb{T}) \quad \text{and} \quad \hat{f}(n)=\int_{-\pi}^{\pi}f(e^{it}) e^{-itn}\: \frac{dt}{2\pi}=0,\quad  \forall n <0. 
 \]
The analytic and the co-analytic projection operators \(P_+,P_- : L^p(\mathbb{T}) \to H^p(\mathbb{T}),\) are defined with
\[
  P_+f(e^{it}) = \sum_{n=0}^{\infty} \hat{f}(n) e^{int}, \quad P_-f(e^{it})=\sum_{n=-\infty}^{-1} \hat{f}(n) e^{int},
\]
where \(f(e^{it})=\sum_{n=-\infty}^{\infty} \hat{f}(n)e^{int}.\) Note that \(P_+ + P_- = I,\) where \(I\) is identity operator. For a function \(f\) defined on the unit disk \(\mathbb{D} = \{ z \in \mathbb{C} \mid \left\lvert z \right\rvert <1 \}\) we denote 
\[
  M_p(f,r) = \left( \int_{T}^{}\left\lvert f(r e^{it}) \right\rvert ^{p}\: \frac{dt}{2\pi}  \right) ^{1/p}.
\]
The harmonic Hardy space \(h^p, 1 \leq  p \leq \infty,\) consists of all harmonic function for which 
\(
  \sup_{0<r<1} M_p(f,r) < \infty.
\)
It is well known that for a harmonic function \(f\) the value of \(M_p(f,r)\) is  increasing with respect to \(r.\) The norm on the harmonic Hardy space \(h^p\) is given by
\[
  \left\lVert f \right\rVert _{p} = \sup_{0<1<r} M_p(f,r) = \lim_{r \to 1-} M_p (f,r).
\]
The analytic Hardy space \(H^p\) is the subspace of \(h^p\) that contains only analytic functions. For all \(f \in H^p\) there exists radial boundary value almost everywhere, i.e. exists \(f^*(e^{it}) = \lim_{r \to 1-} f(re^{it})\) for almost every \(e^{it} \in \mathbb   {T}.\) Also, it is well known that \(f^*\) belongs to the space \(H^p(\mathbb{T}),\) and that \(\left\lVert f^* \right\rVert _{L^p(\mathbb{T})}= \left\lVert f \right\rVert _{H^p}.\) For a  given function \(\varphi \in H^p(\mathbb{T}),\) the Poisson extension of \(\varphi\) gives us a function \(f \in H^p,\) such that \(f^* = \varphi\) almost everywhere on \(\mathbb{T}\). After identifying these spaces, the  analytic projection operator (Riesz projection operator) \(P_+\) can be represented as the Cauchy integral
\[
  P_+f(z)=\frac{1}{2\pi i} \int_{\mathbb{T}}^{}\frac{f(\tau)}{\tau-z}\: d \tau, \quad z \in \mathbb{D}. 
\]

More about the Riesz projection operator and Hardy spaces can be found in \cite{Pavlovic2004,Pavlovic2019,Rudin1987,DurenP1970,Garnett2006}.

 The \(L^p\) boundedness of the Riesz projection operator for \(1<p<\infty\) was first proved by Marcel Riesz in \cite{Riesz1928}. Nowadays, there are many easier proofs of this fact. An interested reader can find  a proof based on Hardy-Stein equality in \cite{SteinP1933}. For \(p=1\) or \(p=\infty\) there is no bounded projection from \(L^p(\mathbb{T})\) to \(H^p(\mathbb{T}),\) \cite{newman1961,rudin1962}.  
 
 Gohberg and Krupnik in \cite{GokhbergKrupnik1968} established the lower estimate of the Riesz projection operator norm by showing that \(\left\lVert P_+ \right\rVert  \geq \frac{1}{\sin (\pi/p)},\) where they also proved that \(\left\lVert P_+ \right\rVert = \frac{1}{\sin (\pi/p)}\), for \(p=2^k,k \in \mathbb{N}.\) For a real-valued function \(f\) such that \(\hat{f}(0)=0\), Verbitsky proved in \cite{verbitsky_1984} that \(\left\lVert P_+ \right\rVert = \max \{ \frac{1}{2\cos \frac{\pi}{2p}}, \frac{1}{2\sin \frac{\pi}{2p}}\}.\) Solving the general problem turned out to be more difficult, and it was finally solved by Hollenbeck and Verbitsky in their seminal paper \cite{Hollenbeck_Verbitsky_2000} from 2000. For complex-valued functions, they showed that \(\left\lVert P_+ \right\rVert = \frac{1}{\sin(\pi/p)}.\) More precisely, the authors proved slightly stronger inequality
 \[
   \left\lVert \max  \left\{ \left\lvert P_+f \right\rvert ,\left\lvert P_- f \right\rvert  \right\}\right\rVert _{L^p(\mathbb{T})} \leq \frac{1}{\sin \frac{\pi}{p}} \left\lVert f \right\rVert _{L^p(\mathbb{T})} 
 \]
 for \(1 <p \leq  2,\)  while the analogous result for \(p \geq 2\) was proven in \cite{Hollenbeck_Verbitsky_2010}. In both papers, the main difficulty was proving some "elementary" inequalities using plurisubharmonic minorants. An interesting general approach to this kind of problems was given by Vasyunin and Volberg \cite{VasyuinVolberg2020}.

 Hollenbeck and Verbitsky in  \cite{Hollenbeck_Verbitsky_2010} posed a problem of finding the optimal constants \(C_{s,p}\) in inequalities 
 \[
   \left\lVert \left( \left\lvert P_+f \right\rvert ^s+ \left\lvert P_-f \right\rvert ^{s} \right) ^{\frac{1}{s}}  \right\rVert _p \leq  C_{s,p} \left\lVert f \right\rVert _p,
 \]
 for \( f \in L^p(\mathbb{T}),1<p<\infty,0<s<\infty.\) The most interesting part of this conjecture is the case where \(0 <s < \max \{ \sec^2 \frac{\pi}{2p},\csc ^2 \frac{\pi}{2p}\},\) where the conjectured values of \(C\) are given by
 \begin{align*}
  C_{s,p} = \frac{2^{1/s}}{2\cos \frac{\pi}{2p}}, & \quad 1<p<2,\; 0<s<\sec^2 \frac{\pi}{2p}\\
  C_{s,p} = \frac{2^{1/s}}{2\sin \frac{\pi}{2p}}, & \quad 2<p<\infty,\; 0 <s<\csc^2 \frac{\pi}{2p}.
 \end{align*}
 Kalaj \cite{Kalaj_2019} proved the conjecture for \(s = 2, 1<p<\infty.\)  An extension of these results was proven by Melentijević and Marković \cite{Melentijevic_2022} where conjecture was shown for \(p \geq 2,0\leq s\leq p;\;1 <p \leq 5/4,0<s\leq 4,\) and \(\frac{5}{4}<p<2,0<s\leq 2.\) In the recent paper \cite{Melentijevic_2023} Melentijević settled down the hypothesis in full generality for \(p \geq 2\;(s>0),\) while for \(1 <p \leq  \frac{4}{3}\) he proved  it for \( s \leq  \sec ^2 \frac{\pi}{2p}.\) Here we will finally address the case where \(4/3 <p<2,\) for \(s \leq  \frac{p}{p-1}.\) The range where \(0 < s \leq \frac{p}{p-1}\) is the best possible where this method works. Even though we believe that the conjecture is true, it cannot be done by this method. It can be seen that the main inequality (\ref{simplified-main-inequality})  does not hold for \(r = 1/2,t=\pi/2,p=3/2, s = 4 = \sec^2 \frac{\pi}{2p}.\) Let us now state our main results.
 \begin{theorem}\label{tm:Ryesz-type-inequalities}
      For \(\frac{4}{3}<p<2\) and \(0 < s \leq \frac{p}{p-1}\) we have that the inequality
    \[
      \left\lVert \left( \left\lvert P_+f \right\rvert ^{s} + \left\lvert P_-f \right\rvert ^{s} \right) ^{1/s}  \right\rVert _p \leq  C_{s,p} \left\lVert f \right\rVert _p,
    \]
    holds for  \(f \in L^p(\mathbb{T})\), where \(C_{s,p}= \frac{2^{\frac{1}{s}}}{2 \cos \frac{\pi}{2p}}.\)
\end{theorem}

One can easily also state the corresponding results for half-line multipliers, half-space multipliers and analytic martingales, by consulting the known results in \cite{Hollenbeck_Verbitsky_2000} and \cite{Melentijevic_2023}.

In  \cite[Section 2, Remark 2.1]{Melentijevic_2022} it is proven that 
\[
  C_{s,p} \geq \frac{2^{1/s}}{2 \cos \frac{\pi}{2p}},
\]
so, our main concern will be to prove the other side of the inequality, which we prove in Section \ref{main-section}. The main method used in this paper is the method of plurisubharmonic minorants. In the paper \cite{Melentijevic_2022} it is proved that the result can be delivered from the elementary inequality (\ref{main_inequality1}). It is easy to transform this inequality into the inequality with two variables (\ref{simplified-main-inequality}), from which can be concluded that it is enough to prove (\ref{main_inequality2}). Given inequality is non-trivial, and we prove it in this paper. Since the expression in (\ref{main_inequality2}) neither increases nor decreases with respect to \(r,\) we replace it with two new expressions that have same values in stationary points and which are monotonical.
\begin{remark}
  Also, it is also worth mentioning that for a real-valued function and a complex valued function the inequality in the Hollenbeck-Verbitsky conjecture holds with same constants for all \(0<s \leq  \max \{ \sec^2\frac{\pi}{2p},\csc ^2 \frac{\omega}{2p}\}.\) However, in the complex case, the given inequality does not hold for \(s > \max \{ \sec ^2 \frac{\pi}{2p}, \csc^2\frac{\pi}{2p}\},\) which means that the constants depends more substantially on \(s.\)
\end{remark}

\begin{theorem}\label{tm:reverse-Riesz-type-inequalities}
  If \(p \in (2,4)\) and \(s \geq \frac{p}{p-1}\) then
  \[
    \left\lVert f \right\rVert _{L^p(\mathbb{T})}\leq 2^{1-\frac{1}{s}} \cos \frac{\pi}{2p}\left\lVert \left( \left| P_+f \right| ^{s}+\left| P_-f \right| ^{s} \right) ^{1/s}   \right\rVert _{L^p(\mathbb{T})}.
  \]
\end{theorem}
 Recent results on this topic include those by Melentijević \cite{melentijevic-reverse-2025} where he addressed the case where \(p \in (1,2] \cup [4,\infty),\) and \(s \geq \min \{ p,\frac{p}{p-1}\}.\) Similar estimates with  additional perturbational constants are proven for \(1<p<2\) in \cite{kalaj2023mrieszconjugatefunction}.
\section{Proof of the Theorem \ref{tm:Ryesz-type-inequalities}}\label{main-section}
In paper \cite[Section 4]{Melentijevic_2022} it is proven that if the inequality holds for some value of \(s=s_0\) then it holds for  \(s <  s_0.\) Therefore, it is enough to prove it for \(s=\frac{p}{p-1}.\) Using the same method that was used in papers \cite{Hollenbeck_Verbitsky_2000,Hollenbeck_Verbitsky_2010,Kalaj_2019,Melentijevic_2022}, it is easy to see that the estimate follows from the inequality:
      \begin{equation}\label{main_inequality1}
        -\left( \frac{\left\lvert z \right\rvert ^s + \left\lvert w \right\rvert ^{s}}{2} \right) ^{\frac{p}{s}}+ \frac{\left\lvert z+\overline{w} \right\rvert ^{p}}{2^{p}\cos^p \frac{\pi}{2p}}-\tan \frac{\pi}{2p}\mathrm{Re}\, (zw)^{\frac{p}{2}}\geq 0,
      \end{equation}
    \(z \in \mathbb{C},w \in \mathbb{C}, p \in \left( \frac{4}{3},2 \right) ,s=\frac{p}{p-1}.\)
    Indeed, the  function \(E(z,w)= \mathrm{Re} \, (zw)^{p/2} \) is plurisubharmonic for \(p \in (1,2]\) (proof of that can be found in \cite[Lemma 2.2]{Hollenbeck_Verbitsky_2000}). When \(f\) is a trigonometric polynomial then the function \(E(P_+f(\xi), \overline{P_-f(\xi)})\) is a subharmonic function for \(\xi \in \mathbb{C},\) which follows from \cite[Theorem 4.13]{MRange1986}. Therefore, 
    \[
      0 = E(P_+f(0),0) = E(P_+f(0),P_+f(0)) \leq  \int_{0}^{2\pi}E(P_+f(e^{it}), \overline{P_-f(e^{it})})\: \frac{dt}{2\pi}.
    \]
    Taking \(z=P_+f(re^{it}),w=\overline{P_-f(re^{it})},\)  and integrating the inequality (\ref{main_inequality1}) over \(\mathbb{T}\) we easily get that the estimate follows when \(f\) is a trigonometric polynomial. Since \(P_+\) and  \(P_-\) are bounded operators for \(1<p<\infty\) we obtain that the estimate holds for all \(f \in L^p(\mathbb{T}).\)

    Dividing the previous equation with \(\max \{ \left\lvert z \right\rvert ^p,\left\lvert w \right\rvert ^p\}\) and writing \(\overline{zw}/\left\lvert z \right\rvert ^2\) in the polar form, we obtain that it is enough to prove 
      \begin{equation}\label{simplified-main-inequality}
        \Phi(r,t)=-\left( \frac{1+r^{s}}{2} \right) ^{\frac{p}{s}}+ \frac{\left( 1+r^2+2r \cos t \right) ^{\frac{p}{2}}}{2^{p}\cos^{p}\frac{\pi}{2p}} - r^{\frac{p}{2}}\tan \frac{\pi}{2p}\cos \frac{tp}{2} \geq 0,
      \end{equation}  
    for \(0\leq  r \leq  1, -\pi \leq  t \leq  \pi.\) It is proved in \cite[Subsection 4.1]{Melentijevic_2022} that the inequality is valid for: \(t \geq \frac{\pi}{p}, r \in [0,1],\)  stationary points of \(\Phi\) in the rectangle \((r,t) \in [0,1]\times [0,\pi/p],\) and the boundary of the rectangle \((r,t) \in [0,1]\times [0,\pi/p],\) except for boundary where \(t=0.\) Here we prove the previous inequality for \(t=0,\) i.e.
\begin{equation}\label{main_inequality2}
  F(r)=-\left( \frac{1+r^s}{2(1+r)^s} \right) ^{\frac{p}{s}} + \frac{1}{2^p \cos^p (\frac{\pi}{2p})}- \left( \frac{\sqrt{r}}{1+r} \right) ^p \tan \left( \frac{\pi}{2p} \right) \geq 0
\end{equation}
for each \(p \in (\frac{4}{3},2),\) \(s=\frac{p}{p-1}\) and \(r \in [0,1].\) 
From 
\[F'(r)=\frac{p}{2}\frac{r^{\frac{p}{2}-1}(1-r)}{(1+r)^{p+1}}\left( \frac{1-r^{s-1}}{1-r}\left( \frac{1+r^s}{2} \right) ^{\frac{p}{s}-1}  \frac{1}{r^{\frac{p}{2}-1}}  - \tan\frac{\pi}{2p} \right) ,\]
we see that \(F'(0+)\cdot F'(1-)\leq  0,\) so we have that the minimum of the function \(F\) occurs  at \(r=0, r=1\) or at points where \(F'(r)=0.\) Points where \(F'(r)=0\)  are those where the equation
\begin{equation}\label{equation_which_we_got_from_derivative-of_f}
  \frac{1-r^{s-1}}{1-r}\left(\frac{1+r^s}{2} \right) ^{\frac{p}{s}-1}  \frac{1}{r^{\frac{p}{2}-1}}  = \tan\frac{\pi}{2p}
\end{equation}
holds. It is sufficient to prove that \(F(0) \geq 0,F(1)\geq 0\) and \(F(r) \geq 0\) for \(r\) which satisfy the equation (\ref{equation_which_we_got_from_derivative-of_f}). Proofs of the inequalities \(F(0) \geq 0, F(1) \geq 0\) can be found in \cite{verbitsky_1984,Melentijevic_2022,Kalaj_2019}. 
The main idea of this paper is to use the identity (\ref{equation_which_we_got_from_derivative-of_f}) in two different ways. In order to prove the inequality \(F(r)\geq 0\) for those \(r\) for which \(F'(r)=0\) we will consider two cases:\\
  \textbf{Case \(r \leq  \frac{1}{2}.\)} In this case, we replace \(\tan \frac{\pi}{2p}\) in the inequality (\ref{main_inequality2}) with the expression on the right side of (\ref{equation_which_we_got_from_derivative-of_f}). After replacement we arrive at the inequality
  \[
    F_1(r)=\frac{1}{2^p\cdot \cos^p \frac{\pi}{2p}} - \left( \frac{1+r^s}{2} \right)^{\frac{p}{s}-1} \frac{1}{(1+r)^{p-1}}\cdot \frac{1-r^s}{2(1-r)} \geq 0; 
    \]
  \textbf{Case \(r \geq  \frac{1}{2}.\)} Similarly, as in the case where \(r\leq  \frac{1}{2},\) one can see that it is enough to prove
  \[
    F_2(r)=\frac{1}{2^p \cdot \cos^p \frac{\pi}{2p}}-\tan\frac{\pi}{2p} \cdot r^{\frac{p}{2}-1}\frac{1}{(1+r)^{p-1}}\cdot \frac{1-r^s}{2(1-r^{s-1})}\geq 0.
    \]

  \subsection{Proof of the first inequality}
  Let us prove that \(F_1\)
  is a  decreasing function with respect to \(r.\) One can easily see that 
    \[
    F_1'(r)=\frac{\left(r^s+1\right)^{\frac{p}{s}-2} \left(p \left( 1-r \right) \left( r^s-1 \right) \left( r^s-r \right)   + 2s r\left( r^{s-1}-r^{s+1} \right) +2r(r^{2s}-1)\right)}{2^\frac{p}{s}(r-1)^2 r(r+1)^p}.
  \]
  \begin{lemma}\label{lm:monotonicity-functions-F1-and-F2}
    The inequalities
    \begin{gather*}
      p \left( 1-r \right) \left( r^s-1 \right) \left( r^s-r \right)   + 2s r\left( r^{s-1}-r^{s+1} \right) +2r(r^{2s}-1)\leq  0, \quad  p \in \left[ 4/3,2 \right] ,s \in [2,4];\\
      p \left( 1-r \right) \left( r^s-1 \right) \left( r^s-r \right)   + 2s r\left( r^{s-1}-r^{s+1} \right) +2r(r^{2s}-1)\geq 0, \quad  p \in [2,4],s \in \left[ 4/3,2 \right] 
    \end{gather*}
    holds for all \(r \in  [0,1].\) Note that in this lemma \(p\) and \(s\) are not restricted with  any constraint.

  \end{lemma}
  \begin{proof}\textbf{The first inequality.}
  The expression is  increasing with respect to \(p.\)  Taking \(p=2,\) we conclude that it is enough to prove that
    \(
      2r^2 \varphi(r)\geq 0,
    \)
    where \(\varphi(r)= (1-r^{2s-2}) +(1-s)(r^{s-2}-r^{s}).\)
    Since
    \(
      \varphi '(r)=(s-1)r^{s-3}(-2r^{s}+s r^2 -s+2),
    \)
    the monotonicity of \(\varphi\) depends on a sign of the expression \(\psi(r)=-2r^s+s r^2 -s+2.\)  As \(\psi '(r)= -2sr^{s-1}+2sr=2s(r-r^{s-1})\geq 0,\) we infer that \(\psi\) increases with respect to \(r.\) Given that \(\psi(1) =0, \) we conclude that \(\psi(r) \leq  0,\) for  \(r \in [0,1],s \in [2,4]. 
    \)
    So, we have that \(\varphi\) 
     decreases with respect to \(r,\) and \(\varphi(1)=0,\) which implies that \(\varphi(r)\geq 0.\)

     \textbf{The second inequality.}
     The expression is increasing with respect to \(p\), therefore, it is sufficient to prove the inequality for \(p=2,\) i.e., \(2r^2 \left( 2 r^{2 s-2}-2 \left(r^2-1\right) (s-1) r^{s-2}-2\right)=r^2 \varphi(r)\geq 0.\) Since \(\varphi ' (r) =  (s-1) r^{s-3} \left(2 r^s-r^2 s+s-2\right),\) and the function \(r \mapsto 2 r^s-r^2 s+s-2\)  is increasing equals zero at \(r=1\) we conclude that \(\varphi\) is decreasing. Given that \(\varphi(1)=0\) we deduce that \(\varphi(r) \geq 0,\) for all \(r \in [0,1].\) 
  \end{proof} 
  Based on the previous lemma, we deduce that the function \(F_1\) is  decreasing on \(r,\) hence, we are reduced  to prove that \(F_1(1/2)\geq 0,\) i.e.,
  \[
    \frac{1}{2^p \cos^p \frac{\pi}{2p}}-\frac{1}{6^{p-1}}\cdot \frac{2^s-1}{(2^s+1)^{2-p}}\geq 0.
  \]
  It is equivalent with
  \[
    -p \log \left( 2\cos \frac{\pi}{2p} \right) +  (p-1) \log 6 - \log \left( 2^s-1 \right)  +(2-p) \log \left( 2^s+1 \right) \geq 0,
  \]
  which, after  using \(p=\frac{s}{s-1},\) can be rewritten as 
    \begin{equation}\label{the_key_form}
      \varphi_1 (s)=-(s-1) \log \left(2^s-1\right)+(s-2) \log \left(2^s+1\right)-s \log \left(2 \sin \left(\frac{\pi }{2 s}\right)\right)+\log (6)\geq 0.
    \end{equation}
  Let's prove that \(\varphi\) is convex for \(s \in [2,4].\) Its second derivative equals to
    \begin{equation}\label{eq:the-second-derivative-of-varphi-1}
        \frac{\pi ^2 \csc ^2\left(\frac{\pi }{2 s}\right)}{4 s^3}+\frac{2^s \log 2 \left(-4^{s+1}+2^{s+1} \log 2+s \left(2^{2 s+1} \log 2+\log 4\right)-4^s \log 8+4-\log 8\right)}{\left(4^s-1\right)^2}.
      \end{equation}
    It is straightforward to show that the function
    \(
       s \mapsto \frac{\pi ^2 \csc ^2\frac{\pi }{2 s}}{4 s^2}
    \)
    is  decreasing with respect to \(s\) by showing that the function \( s \mapsto \frac{\pi \csc \frac{\pi}{2s}}{2s}\) is a non-negative,  decreasing function. Thus, we have 
    \begin{equation}\label{auxiliary-inequality-the-first-inequality}
      \frac{\pi ^2 \csc ^2\frac{\pi }{2 s}}{4 s^2} 
      \geq
      \left.\frac{\pi ^2 \csc ^2\frac{\pi }{2 s}}{4 s^2}\right\vert_{s=4}
      =
      \frac{1}{64} \pi ^2 \csc ^2\frac{\pi }{8}
      =
      \frac{\pi ^2}{16 \left(2-\sqrt{2}\right)}.
    \end{equation}
    Using the substitution \(2^s = t \in [4,16],\) the expression
    \begin{equation}\label{auxiliary_inequality_with_substitution}
    -\frac{ s 2^s \log 2 \left(-4^{s+1}+2^{s+1} \log 2+s \left(2^{2 s+1} \log 2+\log 4\right)-4^s \log 8+4-\log 8\right)}{\left(4^s-1\right)^2}
    \end{equation}
    takes form
    \begin{equation}\label{auxiliary-inequality-after-substitution}
         -\frac{t \log t}{t^2-1} \cdot  \frac{\left(2 \left(t^2+1\right) \log t+t (\log 4-t (4+\log 8))+4-\log 8\right)}{\left(t^2-1\right)}.
    \end{equation}
    Hence, the expression (\ref{auxiliary_inequality_with_substitution})  decreases with respect to \(s\) if and only if the  expression (\ref{auxiliary-inequality-after-substitution})  decreases with respect to \(t.\)

    Since 
    \[
      \frac{\partial }{\partial t} \left( \frac{t \log t}{t^2-1} \right) = \frac{t^2-\left(t^2+1\right) \log t-1}{\left(t^2-1\right)^2},
    \]
    and since \(t^2-\left(t^2+1\right) \log t-1 = t^2 (1-\log t)-1-\log t < 0\) for  \(t > e,\) we see that the function \(t \mapsto \frac{t \log t}{t^2-1}\) is  decreasing. 

    By differentiating the expression
\begin{equation}\label{monotonicity-the-rest-of-the-expression}
       -\frac{2 \left(t^2+1\right) \log t+t (\log 4-t (4+\log 8))+4-\log 8}{\left(t^2-1\right)}
\end{equation}
     we get
    \[
      \frac{-2 t^4+t^3 \log 4-4 t^2 \log 8+8 t^2 \log t+t \log 4+2}{t \left(t^2-1\right)^2},
    \]
    thus, for proving monotonicity of the expression (\ref{monotonicity-the-rest-of-the-expression}) we need to prove 
    \[
      \psi(t)=-2 t^4+t^3 \log 4-4 t^2 \log 8+8 t^2 \log t+t \log 4+2\leq 0
    \]
    for \(t \in[4,16].\) Seeing that,
     \begin{equation*}
       \begin{gathered}
        \psi''(t)= -24 t^2+6 t \log 4+16 \log t+24-8 \log 8 
        \\ \leq  
        -24 \cdot 4 ^2+6 \cdot 16 \log 4 + 16 \log 16 + 24-8 \log 8 
        \\= -360+232 \log 2 <0 , \quad t \in [4,16],
       \end{gathered}
     \end{equation*}   
    we infer that \(\psi'\)  decreases for \(t \in [4,16].\)
    Since
    \begin{equation*}
      \begin{gathered}
        \psi'(4)
        = 
        65 \log 4-480<0
        \quad \text{and} \quad
        \psi(4)= -510+200 \log 2< 0,
      \end{gathered}
    \end{equation*}
    we have \(\psi(t) \leq  0\) for \(t \in [4,16].\)
    As the expression (\ref{monotonicity-the-rest-of-the-expression})
    decreases on \(t\) for its non-negativity,
     it is enough to prove its non-negativity for \(t = 16.\) Taking \(t=16\) we get 
    \[
      \frac{1}{255} (-4+\log 8-16 (\log 4-16 (4+\log 8))-514 \log 16) = 4-\frac{439 \log 2}{85} >0, 
    \]
    so we deduce that the expression (\ref{monotonicity-the-rest-of-the-expression}) is  decreasing and non-negative.

    Using the previous conclusions  we get that the expression (\ref{auxiliary-inequality-after-substitution}) decreases in \(t\) as a product of two non-negative  decreasing expressions. Thus, we have that the expression  (\ref{auxiliary_inequality_with_substitution}) decreases with respect to \(s\) thus giving
    \begin{equation*}
      \begin{gathered}
        \frac{ s 2^s \log 2 \left(-4^{s+1}+2^{s+1} \log 2+s \left(2^{2 s+1} \log 2+\log 4\right)-4^s \log 8+4-\log 8\right)}{\left(4^s-1\right)^2}
        \\ \geq 
        \frac{8}{45} \log 2 (\log 32-12).
      \end{gathered}
    \end{equation*}
    Using the previous inequality and (\ref{auxiliary-inequality-the-first-inequality}), we deduce
    \[
      \varphi_1 ''(s) \geq \frac{1}{s}\cdot \left( \frac{\pi^2}{16 (2-\sqrt{2})} + \frac{8}{45} \log 2 (\log 32-12) \right)  >0,
    \]
    thus, the function \(\varphi_1 \) is convex. 
  Since \(\varphi_1 '(2) = \frac{\pi }{4}-\frac{11}{6}  \log 2-\log 3+\log 5 > 0\) and \(\varphi_1 (2)=0\), we obtain \(\varphi_1 (s) \geq 0,\) for \(s \in [2,4].\)
  \subsection{Proof of the second inequality}
  One can easily show that 
  \[
  F_2'(r)=-\frac{r^{\frac{p}{2}-1} \tan \left(\frac{\pi }{2 p}\right) \left(p \left( 1-r \right) \left( r^s-1 \right) \left( r^s-r \right)   + 2s r\left( r^{s-1}-r^{s+1} \right) +2r(r^{2s}-1)\right)}{4 \left(r-r^s\right)^2(1+r)^p}
\]
  From the Lemma \ref{lm:monotonicity-functions-F1-and-F2} it is straightforward to see that \(F_2'(r)\) is non-negative, thus \(F_2\) is  increasing in \(r \in [1/2,1].\) Therefore, it is enough to prove that \(F_2\left( 1/2 \right) \geq 0,\) i.e.,
  \[
    \frac{1}{2^p \cos^p \frac{\pi}{2p}}- \tan \frac{\pi}{2p}\cdot \frac{2^s-1}{2^s-2}\cdot \frac{2^{\frac{p-2}{2}}}{3^{p-1}}\geq 0, \quad p \in \left[ \frac{4}{3},2 \right] .
  \]
    Using the chain rule, we calculate  the second derivative of the function \( p \mapsto \frac{2^{s(p)}-1}{2^{s(p)}-2},\) where \(s(p)=\frac{p}{p-1}:\)
      \begin{equation*}
        \begin{gathered}
            \frac{2^s (s-1)^3 \log 2 \left(-2^{s+1}-2^s \log 2+s \left(2^s \log 2+\log 4\right)+4-\log 4\right)}{\left(2^s-2\right)^3}.
        \end{gathered}
      \end{equation*}
    So,  to establish the convexity of the function \(p \mapsto \frac{ 2^{s(p)}-1}{2^{s(p)}-2},\) we need to prove  
    \[
      -2^{s+1}-2^s \log 2+s \left(2^s \log 2+\log 4\right)+4-\log 4\geq 0, \quad s \in [2,4],
    \]
     i.e. 
    \(
      -t (2+\log 2)+(t+2) \log t+4-\log 4 \geq 0,  t \in [4,16]
    \) (substitution \(2^s=t\)).
    Since,
      \begin{gather*}
        \left( \frac{\partial }{\partial t}  \right) ^2 \left( -t (2+\log 2)+(t+2) \log t+4-\log 4 \right)  = \frac{t-2}{t^2}>0,\\
        \left. \frac{\partial }{\partial t} \right\vert_{t=4} \left( -t (2+\log 2)+(t+2) \log t+4-\log 4 \right)  = \frac{1}{2}\left( \log 4-1 \right) > 0
      \end{gather*}
    we obtain  \(-t (2+\log 2)+(t+2) \log t+4-\log 4\) is  increasing on \(t\). Its value for \(t=4\) is \(\log 64-4> 0,\) which means that the function \(p \mapsto \frac{2^{s(p)-1}}{2^{s(p)-2}}\) is convex. Therefore, estimating this function by its secant line we have
    \[
    \frac{2^s-1}{2^s-2} \leq \frac{9 p}{14}+\frac{3}{14}
  \]
  holds.
  Hence, in order to prove \(F_2(1/2)\geq 0\) it is enough to show:
    \begin{equation}\label{subinequalities-of-second-inequality}
      \frac{1}{2^p \cos ^p \frac{\pi}{2p}}-\tan \frac{\pi}{2p}\cdot \frac{2^{s}-1}{2^{s}-2}\cdot \frac{2^{\frac{p-2}{2}}}{3^{p-1}}\geq 0
   \quad \text{or} \quad
      \frac{1}{2^p \cos^p \frac{\pi}{2p}}- \tan \frac{\pi}{2p}\left( \frac{9}{14}p +\frac{3}{14} \right) \cdot \frac{2^{\frac{p-2}{2}}}{3^{p-1}} \geq 0.
    \end{equation}
  \subsubsection{Proof of the second inequality for \texorpdfstring{\(p \in [5/3,2] \Leftrightarrow s \in [2,5/2] \)}{p in [5/3,2] or s in[2,5/2]}}
Logarithming the first inequality in  (\ref{subinequalities-of-second-inequality}) and using  \( p = \frac{s}{s-1},\) we are reduced to prove 
\[
  -s \log 2+\log \frac{9}{4}-2 s \log \left(\sin \frac{\pi }{2 s}\right)-2 (s-1) \log \left(\cot \frac{\pi }{2 s}\right)-2 (s-1) \log \left(\frac{2^s-1}{2^s-2}\right)  \geq 0
\] 
for \(s \in [2,5/2].\)

Since the second derivative of the function \(s \mapsto -2 (s-1) \log \left(\frac{2^s-1}{2^s-2}\right)\) equals
    \begin{equation*}
      \begin{gathered}
          -\frac{2^{s+1} \log 2 \left(3\cdot 2^{s+1}-2^{2 s+1}-4^s \log 2+4^s s \log 2-s \log 4-4+\log 4\right)}{\left(2^s-2\right)^2 \left(2^s-1\right)^2},
      \end{gathered}
    \end{equation*}
    proving
    \[
      3\cdot 2^{s+1}-2^{2 s+1}-2^{2 s} \log 2+2^{2 s} s \log 2-s \log 4-4+\log 4 \leq  0,  \quad s \in [2,5/2],
    \]
     we will infer the convexity of function \(s \mapsto -2 (s-1) \log \left(\frac{2^s-1}{2^s-2}\right).\) It is enough to (using substitution \(2^s=t\))
    \begin{equation}\label{izraz1}
     - t^2 (2+\log 2)+\left(t^2-2\right) \log t+6 t-4+\log 4 \leq  0, \quad t \in [4,6].
    \end{equation}
    The second derivative of the LHS is
    \(
      \frac{2}{t^2}+2 \log \left(\frac{t}{2}\right)-1.
    \)
    Since \(t \geq 4,\) we have that \( \log \frac{t}{2}\geq \log 2,\) which gives 
    \(
      \frac{2}{t^2}+2 \log \left(\frac{t}{2}\right)-1 \geq  \frac{2}{t^2}+\log 4 -1 >0.
    \)
    Therefore, expression (\ref{izraz1}) convex on \(t.\) For \(t=4\) and \(t=6\) it takes values \(-12+14 \log 2<0\) and \(34 \log 3-40<0\), thus concluding proof of (\ref{izraz1}).

    The function
 \(
  s \mapsto -2 (s-1) \log \left(\frac{2^s-1}{2^s-2}\right) 
 \)
 is convex, there for estimating it by tangent line we get
 \begin{equation}\label{the-second-inequality-convexity-tangent-line}
   -2 (s-1) \log \left(\frac{2^s-1}{2^s-2}\right) 
   \geq
   -\log \left(\frac{9}{4}\right)  +\log \left(\frac{8 \sqrt[3]{2}}{9}\right) (s-2), \quad s \in [2,5/2].
 \end{equation}

  The second derivative of the expression \(-2 s \log \left(\sin \frac{\pi }{2 s}\right)-2 (s-1) \log \left(\cot \frac{\pi }{2 s}\right)\) is given by 
   \begin{equation*}
     \begin{gathered}
       \frac{\pi  \csc \frac{\pi }{s} \left(s \left(\pi  \csc \frac{\pi }{s}-4\right)-\pi  (s-2) \cot \frac{\pi }{s}\right)}{s^4}.
     \end{gathered}
   \end{equation*} 
 Its negativity follows from the Lemma \ref{auxiliary-lemma-1-second-inequality}.
Therefore, 
\begin{equation*}
   s \mapsto -s \log 2+\log \frac{9}{4}-2 s \log \left(\sin \left(\frac{\pi }{2 s}\right)\right)-2 (s-1) \log \left(\cot \left(\frac{\pi }{2 s}\right)\right)
\end{equation*}
is concave for \(s \in [2,5/2].\) 
Hence, it can be estimated from the below by 
\[
  \log \left(\frac{16}{5} \left(7-3 \sqrt{5}\right)\right)  (s-2 ) + \log\frac{9}{4}.
\]

 Finally, we obtain
 \begin{equation*}
   \begin{gathered}
    -s \log 2+\log \frac{9}{4}-2 s \log \left(\sin \left(\frac{\pi }{2 s}\right)\right)-2 (s-1) \log \left(\cot \left(\frac{\pi }{2 s}\right)\right)-2 (s-1) \log \left(\frac{2^s-1}{2^s-2}\right) 
    \\ \geq 
    (s-2) \left( \log \left(\frac{16}{5} \left(7-3 \sqrt{5}\right)\right) +\log \left(\frac{8 \sqrt[3]{2}}{9}\right) \right)  \geq 0, \quad s \in [2,5/2].
   \end{gathered}
 \end{equation*}
 With this, we have proven the inequality for \(s \in [2,5/2],\) which corresponds to \(p \in [5/3,2].\)
 \subsubsection{Proof of the second inequality for \texorpdfstring{\(p \in [4/3,5/3] \Leftrightarrow s \in [5/2,4]\)}{p in [4/3,5/3] or s in [5/2,4]}}
  In this case we prove the inequality using the sharper inequality (the second inequality in (\ref{subinequalities-of-second-inequality}))
  \[
    \frac{1}{2^p \cos^p \frac{\pi}{2p}}- \tan \frac{\pi}{2p}\left( \frac{9}{14}p +\frac{3}{14} \right) \cdot \frac{2^{\frac{p-2}{2}}}{3^{p-1}} \geq 0.
  \]
  Logarithming the previous inequality and using \(p= \frac{s}{s-1}\) we transform it to the form \(\varphi(s) \geq 0,\) where
  \begin{equation*}
    \varphi(s)=-s \log 2+\log \frac{9}{4}-2 s \log \left(\sin \frac{\pi }{2 s}\right)-2 (s-1) \log \left(\cot \frac{\pi }{2 s}\right)-2 (s-1) \log \left(\frac{12s-3}{14(s-1)}\right).
  \end{equation*}
   We easily find
  \begin{equation*}
    \begin{gathered}
      \varphi''(s)=\frac{1}{s^{3}}\left( \frac{\pi  \csc \frac{\pi }{s} \left(s \left(\pi  \csc \frac{\pi }{s}-4\right)-\pi  (s-2) \cot \frac{\pi }{s}\right)}{s}+\frac{18s^{3}}{(1-4 s)^2 (s-1)} \right).
    \end{gathered}
  \end{equation*}
  Using the Lemma \ref{auxiliary-lemma-1-second-inequality} we obtain that the first term in the parentheses takes the biggest value for \(s=2,\) i.e., 
  \[
    \frac{\pi  \csc \frac{\pi }{s} \left(s \left(\pi  \csc \frac{\pi }{s}-4\right)-\pi  (s-2) \cot \frac{\pi }{s}\right)}{s} \leq  (\pi -4) \pi.
  \]
  On the other side, since 
  \[
      \left( \frac{\partial }{\partial s}  \right) \left( \frac{18s^3}{(1-4 s)^2 (s-1)}\right)  
      = 
      \frac{54 s^2 (1-2 s)}{(s-1)^2 (4 s-1)^3}< 0,
    \]
  we conclude that the second term in the parentheses decreases with respect to \(s\), so it takes the  biggest value for \(s=5/2,\) i.e.,
  \(
     \frac{18s^3}{(1-4 s)^2 (s-1)} \leq  \frac{125}{54}.
  \)
  Since  
  \(
    (\pi -4) \pi +\frac{125}{54}, \leq  0
  \)
  the function \(\varphi\) is concave for \(s \in [5/2,4]\). Noticing that \(\varphi(5/2)=\log \left(\frac{343}{81} \sqrt{\frac{7}{5}-\frac{3}{\sqrt{5}}}\right) > 0\) and 
 \[
  \varphi(4)=
-\log \left(\frac{1265625}{117649}\right) + \log \left( 64(3-2\sqrt{2}) \right)  = \log \left( \frac{117649}{1265625} \cdot 64 (3-2\sqrt{2}) \right) >0,
 \]
 we obtain that \(\varphi(s) \geq 0\) for \(s \in [5/2,4]\Leftrightarrow p \in [4/3,5/3] .\) 
  \begin{lemma}\label{auxiliary-lemma-1-second-inequality}
    The function 
    \[
      s \mapsto\frac{ \pi  \csc \frac{\pi }{s}\left(s \left(\pi  \csc \frac{\pi }{s}-4\right)-\pi  (s-2) \cot \frac{\pi }{s}\right)}{s}
    \]
    is negative and  decreasing for \(s \in [2,4].\)
  \end{lemma}
  \begin{proof}
    Using the monotonicity of the function \(t \mapsto \frac{1}{t},t>0\) and multiplying the expression with \(-1\) we conclude that  it is sufficient to prove that 
    \(
      \pi  \csc (\pi  t) (-1)(\pi  (2 t-1) \cot (\pi  t)+\pi  \csc (\pi  t)-4)
    \)
    is  non-negative and decreases for \(t \in [1/4,1/2].\) Since the function
    \(
      t \mapsto \pi  \csc (\pi  t)
    \)
    is  decreasing and non-negative, it is enough to show the same properties for the function
    \(
      \psi(t)= (-1)(\pi  (2 t-1) \cot (\pi  t)+\pi  \csc (\pi  t)-4).
    \)
    Monotonicity is implied from
    \(\psi '(t)=-\pi  \csc ^2(\pi  t) (-2 \pi  t+\sin (2 \pi  t)-\pi  \cos (\pi  t)+\pi ) \leq 0,
    \)
    i.e.  \(\omega(t)=-2 \pi  t+\sin (2 \pi  t)-\pi  \cos (\pi  t)+\pi \geq 0.\)  Additionally, from
    \[
      \omega ''(t)=  -2\pi t-\sin \left( 2\pi t \right) -\pi\cos \left( \pi t \right)  = \pi ^2 (\pi -8 \sin (\pi  t)) \cos (\pi  t)<0, 
    \]
    we see that the function \(\omega\) is concave for \(t \in [1/4,1/2]\) and since \(\omega(1/4)=-\frac{\pi}{\sqrt{2}}+\frac{\pi }{2}+1 \geq 0\) and \(\omega(1/2)= 0\) we deduce that \(\omega(t) \geq 0.\) Noticing that \(\psi(1/2)=4-\pi >0\) we conclude the proof of the lemma.
  \end{proof}

\section{Proof of the theorem \ref{tm:reverse-Riesz-type-inequalities}}

\begin{remark}
  If we write the previous inequality in the form
  \[
    \left\lVert f \right\rVert  _{L^p(\mathbb{T})}\leq 2 \cos \frac{\pi}{p}\left\lVert \left( \frac{\left| P_+f \right| ^{s}+ \left| P_-f \right| ^s}{2} \right) ^{1/s} \right\rVert _{L^p(\mathbb{T})}
  \]
  it is easy to see that the RHS is decreasing with respect to \(s,\) thus it is enough to prove it for \(s=\frac{p}{p-1}.\)
\end{remark}
In the paper \cite[Section 4]{melentijevic-reverse-2025} it is proved that the main result follows from the inequality 
\[
  \frac{(1+r)^p}{2^p \cos^p \frac{\pi}{2p}} - \left( \frac{1+r^s}{2} \right) ^{p/s}- r^{p/2} \tan \frac{\pi}{2p} \leq  0,
\]
for \(r \in [0,1],p \in [2,4],s=\frac{p}{p-1}.\) Using the same idea as in the first part of paper we can derive the previous inequality from inequalities:
\begin{equation}\label{the-first-inequality-for-reverse-riesz-inequalities}
  F_1(r)=\frac{1}{2^p\cdot \cos^p \frac{\pi}{2p}} - \left( \frac{1+r^s}{2} \right)^{p-2} \frac{1}{(1+r)^{p-1}}\cdot \frac{1-r^s}{2(1-r)} \leq  0, \quad 0 \leq   r \leq \frac{1}{2}, 
\end{equation}
\begin{equation}\label{the-second-inequality-for-reverse-riesz-inequalities}
  F_2(r)=\frac{1}{2^p \cdot \cos^p \frac{\pi}{2p}}-\tan\frac{\pi}{2p}\cdot r^{\frac{p}{2}-1}\frac{1}{(1+r)^{p-1}}\cdot \frac{1-r^s}{2(1-r^{s-1})}\leq  0, \quad  \frac{1}{2} \leq  r\leq  1.
\end{equation}

\subsection{Proof of the first inequality}
Since 
  \[
    F_1'(r)=\frac{\left(r^s+1\right)^{\frac{p}{s}-2} \left(p \left( 1-r \right) \left( r^s-1 \right) \left( r^s-r \right)   + 2s r\left( r^{s-1}-r^{s+1} \right) +2r(r^{2s}-1)\right)}{2^\frac{p}{s}(r-1)^2 r(r+1)^p}
  \]
we can conclude that \(F_1 \) is increasing from the Lemma \ref{lm:monotonicity-functions-F1-and-F2}.
In light of monotonicity of \(F_1 \) it is enough to prove that \(F_1 \left( \frac{1}{2} \right) \leq 0,\) i.e., 
  \[
    \frac{1}{2^p \cos ^p \frac{\pi}{2p}}\leq \frac{1}{6^{p-1}}\cdot \frac{2^s-1}{(2^s+1)^{(2-p)}}.
  \]
Logarithming the previous inequality and using the constraint \(p=\frac{s}{s-1}\) on can easily transform it into
\begin{equation*}
  \varphi_1(s)=-(s-1) \log \left(2^s-1\right)+(s-2) \log \left(2^s+1\right)-s \log \left(2 \sin \left(\frac{\pi }{2 s}\right)\right)+\log (6)\leq  0,
\end{equation*}
for \(s \in [4/3,2].\) The main idea in the rest of the proof is to prove that \(\varphi_1 \) is convex and nonpositive at the endpoints of domain. From the second derivative of \(\varphi_1 \) given in the formula (\ref{eq:the-second-derivative-of-varphi-1}) one can see that to establish convexity, it is enough to prove
\begin{gather*}
  \frac{\pi ^2 \left(2^s+1\right) \csc ^2\left(\frac{\pi }{2 s}\right)}{4 s^2} +\\ \frac{2^s \left(2^s+1\right) s \log (2) \left(-4^{s+1}+2^{s+1} \log (2)+s \left(2^{2 s+1} \log (2)+\log (4)\right)-(4^s+1) \log (8)+4\right)}{\left(4^s-1\right)^2} \geq 0.
\end{gather*}
Lets denote \(\omega_1(s)=  \frac{\pi^2}{4 s^2} \left(2^s+1\right) \csc ^2\left(\frac{\pi }{2 s}\right)\) and the rest of the expression with \(w_2(s).\)
It can be shown that \(\omega_1 \) logarithmic convex, thus, convex. Namely, one can easy see
\begin{equation}\label{eq:logarithmic-convexity-of-expression-for-reverse-riezs-inequalities}
  \frac{d^2}{ds ^2} \log \omega_1 (s)  =  \frac{\frac{2 s^2 \left(2^s s^2 \log ^2(2)+2 \left(2^s+1\right)^2\right)}{\left(2^s+1\right)^2}+\pi  \left[\pi -2 s \sin \left(\frac{\pi }{s}\right)\right]\csc ^2\left(\frac{\pi }{2 s}\right)}{2 s^4}.
\end{equation}
The function \(s \mapsto \frac{2 s^2 \left(2^s s^2 \log ^2(2)+2 \left(2^s+1\right)^2\right)}{\left(2^s+1\right)^2}\) increasing if and only if the function
\[
  t \mapsto \frac{2 \log ^2(t)}{\log ^2(2)} \left(\frac{t \log ^2(t)}{(t+1)^2}+2\right), \quad t \in [2^{4/3},2^{2}]
\]
is increasing. We will prove monotonicity of the previous function writting it as a product of two nonnegative increasing functions. Namely, \(t \mapsto \log ^2 t\) is positive and increasing for \(t \in\left[ 2^{4/3},2^2 \right] \). Also, since
\[
  \frac{d}{dt}\left[ \frac{t \log ^2(t)}{(t+1)^2}+2 \right] =\frac{\log (t) (2 t-t\log (t)+\log (t)+2)}{(t+1)^3} \geq 0
\]
we can conclude monotonicity of the product (the expression \(2 t-t\log (t)+\log (t)+2\) is concave with positive values at the endpoints). Therefore, 
\[
  \frac{2 s^2 \left(2^s s^2 \log ^2(2)+2 \left(2^s+1\right)^2\right)}{\left(2^s+1\right)^2} \geq \frac{64}{81} \left(9+\frac{16 \sqrt[3]{2} \log ^2(2)}{\left(1+2 \sqrt[3]{2}\right)^2}\right) >7.
\]
Lets examine the rest of the numerator in (\ref{eq:logarithmic-convexity-of-expression-for-reverse-riezs-inequalities}). The expression \(s\sin \frac{\pi}{s}\) is increasing with respect to  \(s,\) which means that \(\pi-2s \sin \frac{\pi}{s}\) is decreasing with respect to  \(s.\) When \(\pi - 2s \sin \frac{\pi}{s}\geq 0\) we have that (\ref{eq:logarithmic-convexity-of-expression-for-reverse-riezs-inequalities}) is positive. If \(\pi - 2s \sin \frac{\pi}{s}< 0\), we have that the function \(\csc \frac{\pi}{2s}\left( 2s \sin \frac{\pi}{s}-\pi \right) \) increases, which means that \(\pi  \left(\pi -2 s \sin \left(\frac{\pi }{s}\right)\right) \csc ^2\left(\frac{\pi }{2 s}\right)\) have the least value when \(s=2,\) and its value is \(2 (\pi -4) \pi\geq -6.\)

Now, it is known that the function \(\omega_1 \) is convex, which means that it is either decreasing, decreasing then increasing, or increasing.
The monotonicity of the function coincides with the monotonocity of the function  \(\log \omega_1 (s).\) Taking the first derivative for \(s=\frac{3}{2}\) we have
\[
  \left.\frac{s \left(\frac{2^s s \log (2)}{2^s+1}-2\right)+\pi  \cot \left(\frac{\pi }{2 s}\right)}{s^2}\right|_{s=\frac{3}{2}}=-\frac{4}{3}+\frac{4 \pi }{9 \sqrt{3}}+\frac{\sqrt{2} \log (4)}{2 \sqrt{2}+1}<0
\]
that the function \(\omega_1 \) is decreasing on \(\left[ \frac{4}{3},\frac{3}{2} \right] .\) 

The inequality \(\varphi_1''(s)\geq 0\) on interval \([\frac{4}{3},\frac{3}{2}]\) can be achieved by compasion its values in the endpoints of intervals \(\left[ \frac{4}{3},\frac{7}{5} \right] \) and \(\left[ \frac{7}{5}, \frac{3}{2} \right] .\) Namely, \(\omega_1 \) is decreasing in the interval, and \(\omega_2 \) is increasing (Lemma \ref{lm:monotonicity-of-omega-2}) and \(\omega_1 \left( \frac{7}{5} \right)+\omega_2 \left( \frac{4}{3}\right) \geq 0  \) and \(\omega_1 \left( \frac{3}{2} \right) +\omega_2 \left( \frac{7}{5}\right)\geq 0.  \) In the rest of the interval, we can use that \(\omega_1 \) is convex and \(\omega_2 \) is increasing (Lemma \ref{lm:monotonicity-of-omega-2}), which implies that 
\begin{align*}
  \omega_1 (s)
  & \geq \frac{8\pi^2}{729} \left(s-\frac{3}{2}\right) \underbrace{\left[2 \pi  \left(\sqrt{3}+2 \sqrt{6}\right)+9 \left(-4 \sqrt{2}-2+\sqrt{2} \log (8)\right)\right]}_{\leq 0}+\frac{4}{27}  \left(2 \sqrt{2}+1\right) \pi ^2
  \\ & \geq \frac{4}{729} \pi ^2 \left(18 \sqrt{2}+2 \pi  \sqrt{3}+4 \pi  \sqrt{6}+9+9 \sqrt{2} \log (8)\right)
  \\ & \geq - \frac{3 \left(2 \sqrt{2}+1\right) \log (2) \left(\sqrt{2} \log (256)-56\right)}{49 \sqrt{2}} = -\omega_2 \left( \frac{3}{2} \right) \geq -\omega_2 (s).
\end{align*}
That proves that the function \(\varphi_1\) is convex. As 
\(\varphi_1 \left( \frac{4}{3} \right) <0\) and \(\varphi_1(2)=0\) we have that \(\varphi_1 (s)\leq  0.\)

\begin{lemma}\label{lm:monotonicity-of-omega-2}
  The function \(\omega_2 \) is increasing for \(s \in [4/3,2].\)
\end{lemma}
\begin{proof}
  Using the substitution \(2^{s} = t \) and multiplying with \(-1,\) it is easy to see that it is enough to prove
  \[
     \frac{t \log t }{t^2-1}\cdot\frac{\left(t^2 (4+\log (8))-2 t^2 \log (t)+\log \left(\frac{8}{t^2}\right)-t \log (4)-4\right)}{(t-1)}
  \]
  is decreasing for \(t \in \left[ 5/2,4 \right] .\) Let \(f(t)= \frac{t \log t}{t^2-1},\) and denote the rest of the previous expression by \(g(t).\) Using lemma (\ref{lm:function-g-monotonicity-for-reverse-riesz-inequalities}) we get
  \begin{align*}
    \frac{d}{dt}\left[ f(t) \cdot g(t) \right] = &f'(t)g(t) + f(t) g'(t) \leq f'(t) g\left( \frac{5}{2} \right) + f(t) \leq \frac{25}{2}f'(t)+f(t) =
    \\ = &
    \frac{25 \left(t^2-1\right)+\left(2 t^3-25 t^2-2 t-25\right) \log (t)}{2 \left(t^2-1\right)^2}.
  \end{align*}
  Therefore, if \(\varphi(t)=25 \left(t^2-1\right)+\left(2 t^3-25 t^2-2 t-25\right) \log (t) \leq 0\) the statement of the lemma holds. Given that \(\varphi ^{iv}(t) =\frac{2 \left(6 t^3+25 t^2-2 t+75\right)}{t^4}\geq 0\) and \(\varphi'''\left( \frac{5}{2} \right) = 12 \log \left(\frac{5}{2}\right)-\frac{22}{25} \geq 0\) we conclude that \(\varphi'''(t)\) is a positive function. That means that \(\varphi''(t)\) is increasing. Therefore, \(\varphi'(t)\) first decreases, and then increases which means that \(\varphi '\) attains its maximum values at the endpoints of the interval. As
  \[
    \varphi'\left( \frac{5}{2} \right) = 63-\frac{179}{2} \log \left(\frac{5}{2}\right) \leq  0 \quad \text{i} \quad \varphi'(4)=\frac{495}{4}-106 \log (4) \leq 0
  \]
  we infer that \(\varphi\) is decreasing. Seeing that \(\varphi \left( \frac{5}{2} \right) =\frac{525}{4}-155 \log \left(\frac{5}{2}\right)<0 \) we obtain that \(\varphi(t) \leq 0\) for \(t \in \left[ \frac{5}{2},4 \right] .\)
\end{proof}
\begin{lemma}\label{lm:function-g-monotonicity-for-reverse-riesz-inequalities}
  The function
  \[
    g(t)=\frac{\left(t^2 (4+\log (8))-2 t^2 \log (t)+\log \left(\frac{8}{t^2}\right)-t \log (4)-4\right)}{(t-1)}  
  \]
  is increasing in \(t \in \left[ \frac{5}{2},4 \right] \) and \(0\leq  g'(t) \leq  1.\) Also, \(g\left( \frac{5}{2} \right) \geq 25/2.\)
\end{lemma}
\begin{proof}
  Because
  \[
    g'(t)=\frac{t \left(\log \left(\frac{t^2}{2}\right)-2 (t-2) t \log (t)\right)+t (2 (t-3) t+(t-2) t \log (8)+2)+2}{(t-1)^2 t}
  \]
  to prove monotonicity we will prove \(\varphi(t)=t \left(\log \left(\frac{t^2}{2}\right)-2 (t-2) t \log (t)\right)+t (2 (t-3) t+(t-2) t \log (8)+2)+2 \geq  0.\) Seeing that \(\varphi^{(iv)}(t)=\frac{-12 t^2-8 t+4}{t^3}<0,\) and  \(\varphi'''\left( \frac{5}{2} \right) = -\frac{178}{25}+6 \log (8)-12 \log \left(\frac{5}{2}\right)<0\) we deduce that  \(\varphi'''(t)\leq 0.\) Now, we can deduce that \(\varphi''(t)\) is decreasing, which means that \(\varphi'(t)\) increases and then decreases. That means, that \(\varphi'(t)\) attains a minimal value at the ends of interval. Seeing that
  \[
    \varphi'\left( \frac{5}{2} \right) =9+\frac{163 \log (2)}{4}-\frac{31 \log (5)}{2} >0  \quad \text{i} \quad \varphi'(4)=36 - 29 \log (2) >0
  \]
  we have \(\varphi'(t)\geq 0.\) Thus, \(\varphi(t)\) is increasing function. Because \(\varphi \left( \frac{5}{2} \right) =\frac{1}{8} (6+65 \log (2)-10 \log (5)) \geq 0,\) we establish that  \(g\) is increasing. As a result we have
  \[
    \min_{\frac{5}{2}\leq t\leq 4}g(t) \geq g\left( \frac{5}{2} \right) =14+\log \left(\frac{1048576\ 2^{5/6}}{1953125\ 5^{2/3}}\right) \geq \frac{25}{2}.
  \]
  The inequality \(g'(t)\leq 1\) we will show it in an equivalent form
  \[
    t \left(\log \left(\frac{t^2}{2}\right)-2 (t-2) t \log (t)\right)+t (2 (t-3) t+(t-2) t \log (8)+2)+2\leq (t-1)^2 t,
  \]
  i.e., \( \psi(t)=t^3+t^3 \log (8)-4 t^2-2 t^2 \log (8)-2 (t-2) t^2 \log (t)+t \log \left(\frac{t^2}{2}\right)+t+2\leq  0.\) Since, \(\psi^{iv}(t)=\frac{-12 t^2-8 t+4}{t^3}\leq 0,\) and \(\psi ^{(3)}\left( \frac{5}{2} \right) = -\frac{328}{25}+30 \log (2)-12 \log (5) \leq 0 \) function \(\psi ^{(3)}\) is non-positive, which means that \(\psi ''\) is decreasing. Also, \(\psi ^{(2)}\left( \frac{5}{2} \right) = -\frac{26}{5}+55 \log (2)-22 \log (5)<0,\) which means that the second derivative is negative function, which means that \(\psi\) is concave. Therefore, we can bound \(\psi\) from above by the tangent line in \(t=3,\) i.e.,
  \[
    \psi(t)\leq (t-3) \left[44 \log (2)-28 \log (3)\right]+4 \left(\log \left(\frac{64}{27}\right)-1\right).
  \]
  As the slope of the tangent line is negative, it is sufficient to prove its negativity in \(t=\frac{5}{2},\) i.e., that \(\log (36)-4\leq 0\) which is evident.
\end{proof}
\subsection{Proof of the second inequality}
Since
\[
  F_2'(r)=-\frac{r^{\frac{p}{2}-1} \tan \left(\frac{\pi }{2 p}\right) \left(p \left( 1-r \right) \left( r^s-1 \right) \left( r^s-r \right)   + 2s r\left( r^{s-1}-r^{s+1} \right) +2r(r^{2s}-1)\right)}{4 \left(r-r^s\right)^2(1+r)^p}
\]
\(F_2 \) is  decreasing by Lemma \ref{lm:monotonicity-functions-F1-and-F2}.
Therefore, to prove  the second inequality it is sufficient to prove that \(F_2 \left( \frac{1}{2} \right) \leq 0,\) i.e.,  
    \[
      \frac{1}{2^p \cos ^p \frac{\pi}{2p}}-\tan \frac{\pi}{2p}\cdot \frac{2^{s}-1}{2^{s}-2}\cdot \frac{2^{\frac{p-2}{2}}}{3^{p-1}}\leq  0.
    \]
Taking the logarithm of the previous inequality, we obtain
    \[
      -p \log \left( 2 \cos \frac{\pi}{2p} \right) - \log \left( \tan \frac{\pi}{2p}   \right) - \log \frac{2^s-1}{2^s-2} +\log \sqrt{2} + (1-p) \log \frac{\sqrt{2}}{3} \leq  0.
    \]
Taking into account \( p = \frac{s}{s-1}\) it takes the form
    \[
      \varphi_2(s)=  s \log 2+\log \frac{9}{4}-2 s \log \left(\sin \left(\frac{\pi }{2 s}\right)\right)-2 (s-1) \log \left(\cot \left(\frac{\pi }{2 s}\right)\right)-2 (s-1) \log \left(\frac{2^s-1}{2^s-2}\right)  \leq  0
    \] 
for \(s \in [4/3 ,2].\)
Using Lemma \ref{lm:reverse-riesz-inequalities-concativity-of-the-first-part-of-function-in-second-inequation-} we obtain the following inequality
\[
  -2 s \log \left(2 \sin \left(\frac{\pi }{2 s}\right)\right)-2 (s-1) \log \left(\cot \left(\frac{\pi }{2 s}\right)\right) \leq  -\log (2) (s-2)-\log 4.
\]
Therefore, using the previous inequality we find 
\[
  \varphi_2(s)\leq  \log \left(\frac{9}{4}\right)-2 (s-1) \log \left(\frac{2^s-1}{2^s-2}\right), \quad s \in \left[ \frac{8}{5},2 \right] .
\]
The RHS is convex with respect to \(s\) by Lemma \ref{lm:reverse-riesz-inequalities-convexity-of-the-second-part-of-inequality}. Since the expression on the right-hand side of the previous inequality takes values \(\log \left(\frac{9}{4}\right)-\frac{6}{5} \log \left(\frac{1-2^{8/5}}{2-2^{8/5}}\right) < 0\) and \(0\) at the endpoints of the interval, respectively, we conclude that the inequality holds in this case.
On the interval \(\left[ \frac{4}{3},\frac{8}{5} \right],\) by concavity, holds
\[
  -2 s \log \left(2 \sin \left(\frac{\pi }{2 s}\right)\right)-2 (s-1) \log \left(\cot \left(\frac{\pi }{2 s}\right)\right) \leq -\frac{(2 \pi ) \left(s-\frac{3}{2}\right)}{9 \sqrt{3}}-\log (3),
\]
which implies to
\[
  \varphi_2(s) \leq \frac{\pi  (3-2 s)}{9 \sqrt{3}}+s \log (2)-2 (s-1) \log \left(\frac{2^s-1}{2^s-2}\right)-\log \left(\frac{4}{3}\right), \quad s \in \left[ \frac{4}{3},\frac{5}{8} \right] .
\]
Again, the RHS  is a convex function attains values \(\frac{\pi }{27 \sqrt{3}}+\log (3)+\frac{2}{3}\log \frac{\sqrt[3]{2}-1}{2\sqrt[3]{2}-1}<0\) and \(\frac{1}{5} \left(\log (3888)+6 \log \left(2^{3/5}-1\right)-6 \log \left(2\ 2^{3/5}-1\right)\right)-\frac{\pi }{45 \sqrt{3}} <0\) at the endpoints of the interval \(\left[ \frac{4}{3},\frac{8}{5} \right] \) from which the second inequality follows.
\begin{lemma}\label{lm:reverse-riesz-inequalities-concativity-of-the-first-part-of-function-in-second-inequation-}
  The function
    \[
      \varphi(s)=-2 s \log \left(2 \sin \left(\frac{\pi }{2 s}\right)\right)-2 (s-1) \log \left(\cot \left(\frac{\pi }{2 s}\right)\right)
    \]
  is concave on the interval \(s \in \left[ \frac{4}{3},2 \right] \)
\end{lemma}
\begin{proof}
  From 
  \[
    \varphi''(t)=\frac{\pi  \csc \left(\frac{\pi }{s}\right) \left[s \left(\pi  \csc \left(\frac{\pi }{s}\right)-4\right)-\pi  (s-2) \cot \left(\frac{\pi }{s}\right)\right]}{s^4}
  \]
  we observe that the sign of \(\varphi''\) is determined by the sign of \(s \left(\pi  \csc \left(\frac{\pi }{s}\right)-4\right)-\pi  (s-2) \cot \left(\frac{\pi }{s}\right).\) By multiplying the expression by \(\frac{\sin \frac{\pi}{s}}{s}\) and using the substitution \(t=\frac{1}{s}\) we get that concavity of \(\varphi\) is equivalent with 
  \[
    -4 \sin (\pi  t)+\pi  (2 t-1) \cos (\pi  t)+\pi  \leq 0 \quad t \in \left[ \frac{1}{2},\frac{3}{4} \right] .
  \]
  Since \(\frac{d^2}{dt^2}\left[ -4 \sin (\pi  t)+\pi  (2 t-1) \cos (\pi  t)+\pi \right] =-\pi ^3 (2 t-1) \cos (\pi  t)\geq 0,\) we conclude that the expression \(-4 \sin (\pi  t)+\pi  (2 t-1) \cos (\pi  t)+\pi\) is convex with respect to \(t.\) Its values at the end of the interval are \(\pi -4 \leq 0\) and \(-2 \sqrt{2}-\frac{\pi }{2 \sqrt{2}}+\pi\leq 0,\) which means that the statement of the lemma holds. 
\end{proof}
\begin{lemma} \label{lm:reverse-riesz-inequalities-convexity-of-the-second-part-of-inequality}
  The function 
  \[
    \varphi(s)=-2 (s-1) \log \left(\frac{2^s-1}{2^s-2}\right)
  \]
  is convex for \(s \in \left[ \frac{4}{3},2 \right] .\)
\end{lemma}
\begin{proof}
  Since
  \[
    \varphi''(s)=-\frac{2^{s+1} \log (2) \left(3\cdot 2^{s+1}-2^{2 s+1}-4^s \log (2)+4^s s \log (2)-s \log (4)-4+\log (4)\right)}{\left(2^s-2\right)^2 \left(2^s-1\right)^2}
  \]
  for concavity of \(\varphi\) it is enough to prove \(3\cdot 2^{s+1}-2^{2 s+1}-4^s \log (2)+4^s s \log (2)-s \log (4)-4+\log (4)\leq  0.\) By substitution \(2^s=t\) we get the equivalent inequality
  \[
    \psi(t)=t^2 \log \left(\frac{t}{2}\right)-2 t (t-3) -2 \log (t)-4+\log (4) \leq 0, \quad t \in \left[ 2^{4/3},2^{2} \right] .
  \]
  Since \(\psi^{'''}(t) =\frac{2 \left(t^2-2\right)}{t^3}\geq 0,  \) which means that \(\psi '\) is a convex function. Since \(\psi'(2)=-1\) and \(\psi'(4)=\log (256)-\frac{13}{2}<0\) we can conclude that \(\psi' \leq 0,\) which implies that \(\psi\) is a decreasing function. Because \(\psi(2)=0,\) we conclude that \(\psi(t)\leq 0\) for \(t \in [2,4]\) which proves lemma.
\end{proof}
  \textbf{Conflict of interest.} The autor declares that he has not conflict of interest.
\nocite{*}
\printbibliography
\end{document}